\documentclass[10pt, reqno]{amsart}
\usepackage{color}
\usepackage{amsmath}
\usepackage{amsthm}
\usepackage{amssymb}
\usepackage{amsfonts}
\usepackage{mathrsfs}
\usepackage{amsbsy}
\usepackage{relsize}
\usepackage[colorlinks, linkcolor=red, citecolor=blue, urlcolor=blue, pagebackref, hypertexnames=false]{hyperref}
\usepackage[hyperpageref]{backref}

\usepackage[lmargin=3cm,rmargin=3cm,tmargin=3cm,bmargin=3cm]{geometry}
\newcommand{\beq}{\begin{eqnarray}}
\newcommand{\eeq}{\end{eqnarray}}
\newcommand{\bq}{\begin{equation}}
\newcommand{\eq}{\end{equation}}
\newcommand{\beqn}{\begin{eqnarray*}}
\newcommand{\eeqn}{\end{eqnarray*}}

\newcommand{\vertiii}[1]{{\vert\kern-0.25ex\vert\kern-0.25ex\vert #1
    \vert\kern-0.25ex\vert\kern-0.25ex\vert}}

\newcommand{\ignore}[1]{}

\newtheorem{definition}{Definition}[section]
\newtheorem{proposition}{Proposition}[section]
\newtheorem{theorem}{Theorem}[section]

\newtheorem{remarks}{Remarks}[section]
\newtheorem{lemma}{Lemma}[section]

\numberwithin{equation}{section}
\title[Mixed generalized fractional Brownian motion]{Mixed generalized fractional Brownian motion}

\author[S. Alajmi, E. Mliki]{Shaykhah  Alajmi$^{1}$ and Ezzedine Mliki$^{1, 2}$}
\address{$^{1}$Department of Mathematics, College of Science, Imam Abdulrahman Bin Faisal University, P. O. Box 1982, Dammam, Saudi Arabia.
$^{2}$Basic and Applied Scientific Research Center, Imam Abdulrahman Bin Faisal University, P.O. Box 1982, Dammam, 31441, Saudi Arabia. }
\email{\sl sho192010@hotmail.com}
\email{\sl ermliki@iau.edu.sa}

\begin{document}
\begin{abstract} To extend several known centered Gaussian processes,  we introduce a new centred mixed self-similar Gaussian process called the mixed generalized fractional Brownian motion, which could serve as a good model for a larger class of natural phenomena. This process generalizes both the  well-known  mixed fractional Brownian motion introduced  by  Cheridito \cite{Cherid} and the generalized fractional Brownian motion introduced by Zili  \cite{Zili}. We study its main stochastic properties, its non-Markovian and non-stationarity characteristics and the conditions under which it is not a semimartingale. We prove  the  long range dependence properties of this process.
\end{abstract}





\maketitle

\section{Introduction}

 Fractional Brownian motion on the whole real line (fBm  for short)  $B^{H}=\{B_{t}^{H},\, t\in\mathbb{R}\}$ of Hurst parameter $H$ is the best known centered Gaussian process with long-range dependence. Its covariance function is
\begin{eqnarray}\label{C1}
	\mbox{Cov}(B_{t}^{H},\,B_{s}^{H})= \frac{1}{2} [|t|^{2H}+|s|^{2H}- |t-s|^{2H}]
\end{eqnarray}
where $H$ is a real number in $(0, 1)$ and the case $H=\frac{1}{2}$ corresponds to the Brownian motion. It is the unique continuous Gaussian process starting from zero, the self-similarity and stationarity  of the increments are two main properties for which fBm enjoyed successes as modeling tool in finance and telecommunications.  Researchers have applied  fractional Brownian motion  to a wide range of problems, such as bacterial colonies, geophysical data, electrochemical deposition, particle diffusion, DNA sequences and stock market indicators \cite{AlexY} and \cite{Omer}. In particular, computer science applications of fBm include modeling network traffic and generating
graphical landscapes \cite{Nan} and  \cite{TWS}. The fBm was investigated in many papers (e.g. \cite{ AlexYES, Alo, GaoH, Dhhd, ManVa}). The main difference between fBm and regular Brownian motion is that the increments in Brownian motion are independent, increments for fBm are not.

In \cite{Bojd}, the authors suggested another kind of extension of the Brownian motion, called the sub-fractional
Brownian motion (sfBm for short), which preserves most properties of the fBm, but not the stationarity of the increments. It is a centered Gaussian process  $\xi^H=\left\lbrace \xi_t^H, \ t\geq 0\right\rbrace,$ defined  by:
\begin{eqnarray}
	\xi_t^H=\frac{B_t^H+B_{-t}^H}{\sqrt{2}}, \quad t\geq 0,
\end{eqnarray}
where $H\in(0,1)$. The case $H=\frac{1}{2}$ corresponds to the Brownian motion.

The sfBm is intermediate between Brownian motion and fractional Brownian motion
in the sense that it has properties analogous to those of fBm, self-similarity,  not Markovian but  the increments
on nonoverlapping intervals are more weakly correlated, and their covariance decays
polynomially at a higher rate in comparison with fBm (for this reason in  \cite{Bojd} is called sfBm). So the sfBm does not generalize the fBm.
The sfBm was investigated in many papers (e.g. \cite{BardBa, Bojd, Sgh, Tud}).

An extension  of the sfBm was introduced by Zili in \cite{ZiliMS} as
a linear combination of a finite number of independent sub-fractional Brownian motions.
It was called the mixed sub-fractional Brownian motion (msfBm for short). The msfBm
is a centered mixed self-similar Gaussian process and does not have stationary increments.
The msfBm do not generalize the fBm.

In \cite{Zili} Zili  introduced new model called the generalized  fractional Brownian motion (gfBm for short) which is an extension of both sub-fractional Brownian motion and fractional Brownian motion.  A gfBm with parameters $a, b,
	\ and\ H$, is a process $ Z^H=\left\lbrace Z_t^H(a,b), t\geq0\right\rbrace $
	defined by
	\begin{eqnarray}
		 Z_t^H(a,b)=aB_t^H+bB_{-t}^H, \quad t\geq0
	\end{eqnarray}
The gfBm was investigated in \cite{ZiliM} and  \cite{ElNot}. The gfBm generalize the sfBm but not the the  mixed fractional Brownian motion.

The  mixed fractional Brownian motion (mfBm for short)  is a linear combination between a Brownian motion and an independent fractional Brownian motion of Hurst parameter $H$. It was introduced by Cheridito \cite{Cherid} to present a stochastic model of
the discounted stock price in some arbitrage-free and complete financial markets. The mfBm is a centered Gaussian process starting from
zero with covariance function
\begin{eqnarray}
	\mbox{Cov}(N_t^H(a, b), N_s^H(a, b))=a^{2}(t\wedge s)+\frac{b^{2}}{2}\left(t^{2H}+s^{2H}-|t-s|^{2H}\right),
\end{eqnarray}
with $H\in(0,1).$ When $a=1$ and $b=0,$ the mfBm is the Brownian motion and when $a=0$ and $b=1,$  is the fBm.
 We refer also to \cite{ElNo, Cherid, ZiliMM, Thale} for further information on this process.

In this paper, we introduce a new stochastic model, which we call the mixed generalized fractional Brownian motion.

\begin{definition} A mixed generalized fractional Brownian motion (mgfBm for short) of parameters $a, b, c$ and $H\in (0, \, 1)$ is a centred
Gaussian  process $M^{H}(a, b, c)=\{ M_{t}^{H}(a, b, c),\, t\geq 0\},$ defined on a probability space $(\Omega, \mathcal{F}, \mathbb{P}),$  with the covariance function

\begin{eqnarray}\label{C2}
C(t,s) =a^2(t\wedge s)+\frac{(b+c)^2}{2}(t^{2H}+s^{2H})-bc(t+s)^{2H}-\frac{(b^2+c^2)}{2}|t-s|^{2H}
 \end{eqnarray}
 where $t\wedge s=\frac{1}{2}\left( t+s-|t-s|\right)$.
\end{definition}

The mgfBm is completely different from all the extensions mentioned above. The  process $M^{H}(a, b, c) $ is motivated by the fact that  this process already introduced for specific values of $a$, $b$ and $c$.  Indeed  $M^{H}(a, b, 0) $ is the mixed  fractional Brownian motion and $M^{H}(0, b, c), $ is the generalized fractional Brownian motion. This why we will name $M^{H}(a, b, c) $ the mixed generalized fractional Brownian motion. It allows to deal with a larger class of modeled natural phenomena, including those
with stationary or non-stationary increments.

Our goal is to study the main stochastic properties of this new model, paying attention to the long-range dependence,  self-similarity,
increment stationary, Markovity and  semi-martingale properties.

\section{The main  properties}
Existence of the mixed generalized fractional Brownian motion  $M^{H}(a, b, c)$ for any $H\in (0, 1)$ can be shown in the following way, consider the process
\begin{eqnarray}\label{MH}
M_{t}^{H}(a, b, c) =aB_{t}+bB_{t}^{H}+cB_{-t}^{H},\quad t\geq0,
\end{eqnarray}
where $B=\{B_{t}, \,t\in\mathbb{R} \}$ is a Brownian motion and $B^{H}=\{B_{t}^{H},\, t\in\mathbb{R} \}$ is an independent fractional Brownian
motion with Hurst parameter $H\in(0, 1).$  \\

Using (\ref{C1}) and since  $B$ and $B^{H}$  are independent we obtain the following lemma.

\begin{lemma}
For all $s,\  t\geq0$, the process (\ref{MH})  is a centered Gaussian process with covariance function given by (\ref{C2}).
	\end{lemma}

	\begin{proof}
	  Let $s,\  t\geq0$ and $C(t, s)= Cov\left( M_t^{H}(a, b, c), M_s^{H}(a, b, c)\right).$ Then
		\begin{eqnarray*}
			C(t, s)&=&Cov\left[ \left(aB_{t}+bB_{t}^{H}+cB_{-t}^{H}\right),\left( aB_{s}+bB_{s}^{H}+cB_{-s}^{H} \right)\right]\\
			&=&a^2(t\wedge s)+b^2Cov(B_t^H,B_s^H)+bc[Cov(B_t^H,B_{-s}^H)]+
			cb[Cov(B_{-t}^H,B_s^H)]\\&+&c^2Cov(B_{-t}^H,B_{-s}^H)\\
			&=&a^2(t\wedge s)+\frac{b^2}{2}\left(t^{2H}+s^{2H}-|t-s|^{2H} \right)+\frac{bc}{2}\left( t^{2H}+s^{2H}-|t+s|^{2H} \right)\\
			&+&\frac{cb}{2}\left( t^{2H}+s^{2H}-|-(t+s)|^{2H} \right)+\frac{c^2}{2}\left( t^{2H}+s^{2H}-|-(t-s)|^{2H} \right)\\
			&=& a^2(t\wedge s)+\frac{b^2}{2}t^{2H}+\frac{b^2}{2}s^{2H}-\frac{b^2}{2}|t-s|^{2H}+\frac{bc}{2}t^{2H}+\frac{bc}{2}s^{2H}\\
			&=& a^2(t\wedge s)+\frac{(b+c)^2}{2}(t^{2H}+s^{2H})-bc|t+s|^{2H}-\frac{(b^2+c^2)}{2}|t-s|^{2H}.
		\end{eqnarray*}

Hence  the covariance function of the  process (\ref{MH})  is precisely $C(t,s)$ given by (\ref{C2}).
Therefore the $M^{H}(a, b, c)$ exists.
\end{proof}

	\begin{remarks}
		Some special cases of the mixed generalized fractional Brownian motion:
	\begin{enumerate}	
  \item If $a=0, b=1, c=0,$ then $M^{H}(0, 1, 0)$ is a fBm.
\item If $a=0, b=c=\frac{1}{\sqrt{2}},$ then $M^{H}(0, \frac{1}{\sqrt{2}}, \frac{1}{\sqrt{2}})$
is a  sfBm.
\item If  $a=1, b=0, c=0,$ then $M^{H}(1, 0, 0)$ is a  Bm.
\item If $a=0,$ then $M^{H}(0, b, c),$ is a  gfBm.
\item If $c=0,$ then $M^{H}(a, b, 0),$  is a mfBm.
\item If $b=c,$ then $M^{H}(a, \frac{b}{\sqrt{2}}, \frac{b}{\sqrt{2}}),$  is the smfBm.
\end{enumerate}
	\end{remarks}
So the mixed generalized fractional Brownian motion is, at the same, a generalization of the fractional Brownian motion,  sub-fractional Brownian motion, the sub-mixed fractional  Brownian motion, generalized fractional Brownian motion,  mixed fractional Brownian motion and of course of the standard Brownian motion.

\begin{proposition}\label{p1}
		The mgfBm   satisfies the following properties:
\begin{enumerate}
\item For all $ t\geq0$,
\begin{eqnarray*}E\left(M_t^H(a,b,c) \right)^2 =a^2t+\left(b^2+c^2-(2^{2H}-2)bc \right)t^{2H} .\end{eqnarray*}
\item Let $0\leq s<t$. Then
	\begin{eqnarray*}
		E \left(M_t^H(a,b,c)-M_s^H(a,b,c)\right)^2 &=&a^2|t-s|-2^{2H}bc(t^{2H}+s^{2H})\\&+&(b^2+c^2)|t-s|^{2H}+2bc|t+s|^{2H}.
	\end{eqnarray*}

\item We have for all   $0\leq s<t$,
	\begin{eqnarray*}
		a^2(t-s)+\gamma_{(b, c, H)}(t-s)^{2H} \leq E \left(M_t^H(a,b,c)-M_s^H(a,b,c)\right)^2 \leq a^2(t-s)+\nu_{(b, c, H)}(t-s)^{2H}
	\end{eqnarray*}
where
\begin{eqnarray*}
\gamma_{(b, c, H)}=(b^2 + c^2-2bc(2^{2H-1}-1))\,\mathbf{1}_{\mathcal{C}}(b,c,H)+(b^2 + c^2)\,\mathbf{1}_{\mathcal{D}}(b,c,H),
	\end{eqnarray*}

\begin{eqnarray*}
\nu_{(b, c, H)}=(b^2 + c^2)\,\mathbf{1}_{\mathcal{C}}(a,b,H)+(b^2 + c^2-2bc(2^{2H-1}-1))\,\mathbf{1}_{\mathcal{D}}(b,c,H),
	\end{eqnarray*}
\begin{eqnarray*}
\mathcal{C}=\{(b,c,H)\in \mathbb{R}^{2}\times]0, 1[;  \;(H>\frac{1}{2}, \;bc\geq0 )\; or \;(H<\frac{1}{2}, \;bc\leq 0)\},
	\end{eqnarray*}
and
\begin{eqnarray*}
\mathcal{D}=\{(b,c,H)\in \mathbb{R}^{2}\times]0, 1[;\;  (H>\frac{1}{2}, \;bc\leq0)\;or\; (H<\frac{1}{2},\;bc\geq 0)\}.
	\end{eqnarray*}
\end{enumerate}

\end{proposition}
\begin{proof}
\begin{enumerate}

\item It is a direct   consequence  of   (\ref{C2}).

\item Let $0\leq s<t$ and $\alpha(t, s)=E\left( M_t^H(a,b,c)-M_s^H(a,b,c))^2\right)$. Then
	\begin{eqnarray*}
	\alpha(t, s)
		&=&E \left( M_t^H(a,b,c)\right) ^2\ 	+E\left( M_s^H(a,b,c)\right) ^2
-2E\left( M_t^H(a,b,c)M_s^H(a,b,c)
  \right) \\&=&a^2t+b^2t^{2H}+2bct^{2H}-2^{2H}bct^{2H}+c^2t^{2H}
+a^2s+b^2s^{2H}+2bcs^{2H}-2^{2H}bcs^{2H}\\&+&c^2s^{2H}-2a^2(t\wedge s)-b^2t^{2H}-b^2s^{2H}+b^2|t-s|^{2H}-bct^{2H}-bcs^{2H}+bc|t+s|^{2H}\\&-&cbt^{2H}-cbs^{2H}+cb|t+s|^{2H}-c^2t^{2H}-c^2s^{2H}+c^2|t-s|^{2H}\\
&=&a^2(t+s)-2^{2H}bc(t^{2H}+s^{2H})-2a^2(t\wedge s)+(b^2+c^2)|t-s|^{2H}+2bc|t+s|^{2H}    \\
&=&  a^2|t-s|-2^{2H}bc(t^{2H}+s^{2H})+(b^2+c^2)|t-s|^{2H}+2bc|t+s|^{2H}.
\end{eqnarray*}

\item It is a direct consequence of the second item of    Proposition \ref{p1} and  Lemma 3  in \cite{Zili}.
\end{enumerate}
 \end{proof}	

\begin{proposition}\label{pr3}
 For all $(a, b, c)\in\mathbb{R}^{3}\setminus\{(0, 0,0)\}$ and $H\in(0,1)\setminus\{\frac{1}{2}\}$  the mgfBm is not a self-similar process.
 \end{proposition}
\begin{proof}
This follows from the fact that, for fixed $h>0$, the processes $\left\lbrace M_{ht}^H(a,b,c),\ t\geq0\right\rbrace $
and  $\left\lbrace h^HM_t^H(a,b, c),\ t\geq0\right\rbrace $ are Gaussian, centered, but don't have the same covariance function. Indeed
\begin{eqnarray*}
 C\left( ht, hs\right)
&=&
a^2(ht\wedge hs)+\frac{b^2}{2}\left((ht)^{2H}+(hs)^{2H}-|ht-hs|^{2H} \right)\\&+&\frac{bc}{2}\left( (ht)^{2H}+(hs)^{2H}-|ht+hs|^{2H} \right)
\\&+&\frac{cb}{2}\left( (ht)^{2H}+(hs)^{2H}-|-(ht+hs)|^{2H} \right)\\&+&\frac{c^2}{2}\left( (ht)^{2H}+(hs)^{2H}-|-(ht-hs)|^{2H} \right)\\
&=& a^2(ht\wedge hs)+\frac{b^2}{2}(ht)^{2H}+\frac{b^2}{2}(hs)^{2H}-\frac{b^2}{2}|ht-hs|^{2H}\\&+&\frac{bc}{2}(ht)^{2H}+\frac{bc}{2}(hs)^{2H}-
\frac{bc}{2}|ht+hs|^{2H}+\frac{bc}{2}(ht)^{2H}+\frac{bc}{2}(hs)^{2H}\\&-&\frac{bc}{2}|ht+hs|^{2H}+\frac{c^2}{2}(ht)^{2H}+\frac{c^2}{2}(hs)^{2H}-
\frac{c^2}{2}|ht-hs|^{2H}\\
&=& a^2h(t\wedge s)+h^{2H}\frac{(b+c)^2}{2}\left((t)^{2H}+(s)^{2H}\right)-bch^{2H}|t+s|^{2H}-h^{2H}\frac{(b^2+c^2)}{2}|t-s|^{2H}.
\end{eqnarray*}
On the other hand,
\begin{eqnarray*}
Cov\left( h^HM_t^{H}(a,b, c),h^HM_s^{H}(a,b, c)\right)
&=&h^{2H}Cov\left( M_t^{H}(a,b, c),M_s^{H}(a,b, c)\right)\\
&=& a^2h^{2H}(t\wedge s)+h^{2H}\frac{(b+c)^2}{2}\left((t)^{2H}+(s)^{2H}\right)-bch^{2H}|t+s|^{2H}\\&-&h^{2H}\frac{(b^2+c^2)}{2}|t-s|^{2H}.
\end{eqnarray*}
Then the mgfBm is not a self-similar process for all $(a, b, c)\in\mathbb{R}^{3}\setminus\{(0, 0,0)\}.$
\end{proof}

\begin{remarks}
		As a consequence of Proposition \ref{pr3}, we see that:
	\begin{enumerate}	
  \item $M^{H}(0, b, c)$ is a self-similar process for all $( b, c)\in \mathbb{R}^{2}.$
\item $M^{\frac{1}{2}}(a, b, c)$ is a self-similar process for all $( a, b, c)\in \mathbb{R}^{3}.$
\end{enumerate}
	\end{remarks}
Now, we will study the Markovian property.

\begin{theorem}
	Assume  $H\in(0,1)\setminus \left\lbrace {\frac{1}{2}}\right\rbrace , a\in \mathbb{R} \ and \ (b,c) \in \mathbb{R}^2\setminus\left\lbrace (0,0)\right\rbrace.$
Then $M^H(a,b,c) $  is not a Markovian process.
	\end{theorem}
	\begin{proof}
	The process $M^{H}(a, b, c)$ is a centered Gaussian.
Then, if $M_t^H(a, b, c)$ is a Markovian process, according to Revuz and Yor \cite{Revu}, for all $s<t<u$, we would have
\begin{eqnarray*}
	C(s,u) C(t,t)=C(s,t) C(t,u).
\end{eqnarray*}
We will only prove the theorem in the case where $a\neq 0$, the result with $a= 0$ is known in \cite{Zili}.  For the proof we follow the proof of Proposition 1 given in  \cite{Zili}.\\
Using Proposition \ref{p1} we get
\begin{eqnarray*}
C(s,u)&=&a^2 s+\frac{(b+c)^2}{2}(u^{2H}+s^{2H})-bc|u+s|^{2H}-\frac{(b^2+c^2)}{2}|u-s|^{2H},\\
C(t,t)&=&a^2t+\left(b^2+c^2-(2^{2H}-2)bc \right) t^{2H},\\
C(s,t)&=&a^2s+\frac{(b+c)^2}{2}(t^{2H}+s^{2H})-bc|t+s|^{2H}-\frac{(b^2+c^2)}{2}|t-s|^{2H},\\
C(t,u)&=&a^2 t+\frac{(b+c)^2}{2}(u^{2H}+t^{2H})-bc|u+t|^{2H}-\frac{(b^2+c^2)}{2}|u-t|^{2H}.
\end{eqnarray*}
In the particular case where $1<s=\sqrt{t}<t<u=t^2$, we have
\begin{eqnarray*}
	C(\sqrt{t},t^2)&=&a^2 t^{\frac{1}{2}}+\frac{(b+c)^2}{2}(t^{4H}+t^{H})-bc|t^2+t^{\frac{1}{2}}|^{2H}-\frac{(b^2+c^2)}{2}|t^2-t^{\frac{1}{2}}|^{2H},\\
	C(t,t)&=&a^2t+\left(b^2+c^2-(2^{2H}-2)bc \right) t^{2H},\\
	C(\sqrt{t},t)&=&a^2 t^{\frac{1}{2}}+\frac{(b+c)^2}{2}(t^{2H}+t^{H})-bc|t+t^{\frac{1}{2}}|^{2H}-\frac{(b^2+c^2)}{2}|t-t^{\frac{1}{2}}|^{2H},\\
	C({t},t^2)&=&a^2 t+\frac{(b+c)^2}{2}(t^{4H}+t^{2H})-bc|t^2+t|^{2H}-\frac{(b^2+c^2)}{2}|t^2-t^{\frac{1}{2}}|^{2H}.
\end{eqnarray*}
Then by using that,
\begin{eqnarray*}
	C(\sqrt{t},t^2) C(t,t)=C(\sqrt{t},t) C(t,t^2),
\end{eqnarray*}
we have
\begin{eqnarray*}
& &\left[a^2 t^{\frac{1}{2}}+\frac{(b+c)^2}{2}(t^{4H}+t^{H})-bc|t^2+t^{\frac{1}{2}}|^{2H}-\frac{(b^2+c^2)}{2}|t^2-t^{\frac{1}{2}}|^{2H} \right]\\
&\times&\left[ a^2t+\left(b^2+c^2-(2^{2H}-2)bc \right) t^{2H}\right] \\
&=&\left[ a^2 t^{\frac{1}{2}}+\frac{(b+c)^2}{2}(t^{2H}+t^{H})-bc|t+t^{\frac{1}{2}}|^{2H}-\frac{(b^2+c^2)}{2}|t-t^{\frac{1}{2}}|^{2H}\right]  \\&\times&\left[ a^2 t+\frac{(b+c)^2}{2}(t^{4H}+t^{2H})-bc|t^2+t|^{2H}-\frac{(b^2+c^2)}{2}|t^2-t^{\frac{1}{2}}|^{2H}\right] .	
	\end{eqnarray*}
It follows that
	\begin{eqnarray*}
		& &\left[a^2 t^{\frac{1}{2}}+t^{4H}\left( \frac{(b+c)^2}{2}(1+t^{-3H})-bc|1+t^{-\frac{3}{2}}|^{2H}-\frac{(b^2+c^2)}{2}|1-t^{-\frac{3}{2}}|^{2H} \right) \right]\\
		&\times&\left[ a^2t+\left(b^2+c^2-(2^{2H}-2)bc \right) t^{2H}\right] \\
		&=&\left[ a^2 t^{\frac{1}{2}}+t^{2H}\left( \frac{(b+c)^2}{2}(1+t^{-H})-bc|1+t^{-\frac{1}{2}}|^{2H}-\frac{(b^2+c^2)}{2}|1-t^{-\frac{1}{2}}|^{2H}\right) \right]  \\&\times&\left[ a^2 t+t^{4H}\left( \frac{(b+c)^2}{2}(1+t^{-2H})-bc|1+t^{-1}|^{2H}-\frac{(b^2+c^2)}{2}|1-t^{-1}|^{2H}\right) \right].
	\end{eqnarray*}
Hence						
	\begin{eqnarray*}
a^4 t^{\frac{3}{2}}&+&a^2\left(b^2+c^2-(2^{2H}-2)bc \right)t^{2H+\frac{1}{2}}\\
&+&a^2\left( \frac{(b+c)^2}{2}(1+t^{-3H})-bc|1+t^{-\frac{3}{2}}|^{2H}-\frac{(b^2+c^2)}{2}|1-t^{-\frac{3}{2}}|^{2H} \right)t^{4H+1}\\
&+&\left( \frac{(b+c)^2}{2}(1+t^{-3H})-bc|1+t^{-\frac{3}{2}}|^{2H}-\frac{(b^2+c^2)}{2}|1-t^{-\frac{3}{2}}|^{2H} \right)\\
&\times&\left(b^2+c^2-(2^{2H}-2)bc \right)t^{6H}\\
&=&a^4t^{\frac{3}{2}}\\&+&a^2\left( \frac{(b+c)^2}{2}(1+t^{-2H})-bc|1+t^{-1}|^{2H}-\frac{(b^2+c^2)}{2}|1-t^{-1}|^{2H}\right)t^{4H+\frac{1}{2}}\\
&+&a^2\left( \frac{(b+c)^2}{2}(1+t^{-H})-bc|1+t^{-\frac{3}{2}}|^{2H}-\frac{(b^2+c^2)}{2}|1-t^{-\frac{3}{2}}|^{2H}\right)t^{2H+1}\\
&+&\left( \frac{(b+c)^2}{2}(1+t^{-H})-bc|1+t^{-\frac{3}{2}}|^{2H}-\frac{(b^2+c^2)}{2}|1-t^{-\frac{3}{2}}|^{2H}\right)\\
&\times&\left( \frac{(b+c)^2}{2}(1+t^{-2H})-bc|1+t^{-1}|^{2H}-\frac{(b^2+c^2)}{2}|1-t^{-1}|^{2H}\right)t^{6H}
	\end{eqnarray*}
Take $t^{6H}$ as a common factor, we get
	\begin{eqnarray*}
&&a^2\left(b^2+c^2-(2^{2H}-2)bc \right)t^{-4H+\frac{1}{2}}\\
	&+&a^2\left( \frac{(b+c)^2}{2}(1+t^{-3H})-bc|1+t^{-\frac{3}{2}}|^{2H}-\frac{(b^2+c^2)}{2}|1-t^{-\frac{3}{2}}|^{2H} \right)t^{-2H+1}\\
	&+&\left( \frac{(b+c)^2}{2}(1+t^{-3H})-bc|1+t^{-\frac{3}{2}}|^{2H}-\frac{(b^2+c^2)}{2}|1-t^{-\frac{3}{2}}|^{2H} \right)\\
	&\times&\left(b^2+c^2-(2^{2H}-2)bc \right)\\
	&=&a^2\left( \frac{(b+c)^2}{2}(1+t^{-2H})-bc|1+t^{-1}|^{2H}-\frac{(b^2+c^2)}{2}|1-t^{-1}|^{2H}\right)t^{-2H+\frac{1}{2}}\\
	&+&a^2\left( \frac{(b+c)^2}{2}(1+t^{-H})-bc|1+t^{-\frac{1}{2}}|^{2H}-\frac{(b^2+c^2)}{2}|1-t^{-\frac{1}{2}}|^{2H}\right)t^{-4H+1}\\
	&+&\left( \frac{(b+c)^2}{2}(1+t^{-H})-bc|1+t^{-\frac{1}{2}}|^{2H}-\frac{(b^2+c^2)}{2}|1-t^{-\frac{1}{2}}|^{2H}\right)\\
	&\times&\left( \frac{(b+c)^2}{2}(1+t^{-2H})-bc|1+t^{-1}|^{2H}-\frac{(b^2+c^2)}{2}|1-t^{-1}|^{2H}\right).
\end{eqnarray*}

Therefore
\begin{eqnarray*}
	&&a^2\left(b^2+c^2-(2^{2H}-2)bc \right)t^{-4H+\frac{1}{2}}\\
&+&a^2[ \frac{1}{2}(b+c)^2\left( 1+t^{-3H}\right) -bc\left(1+2Ht^{-\frac{3}{2}}+H(2H-1)t^{-3}+\circ(t^{-3}) \right)\\
	&-&\frac{1}{2}(b^2+c^2)\left(1-2Ht^{-\frac{3}{2}}+H(2H-1)t^{-3}+\circ(t^{-3}) \right) ] t^{-2H+1}    \\
	&+&[ \frac{1}{2}(b+c)^2\left( 1+t^{-3H}\right) -bc\left(1+2Ht^{-\frac{3}{2}}+H(2H-1)t^{-3}+\circ(t^{-3}) \right)\\
	&-&\frac{1}{2}(b^2+c^2)\left(1-2Ht^{-\frac{3}{2}}+H(2H-1)t^{-3}+\circ(t^{-3}) \right) ] \\
	&\times &[b^2+c^2-(2^{2H}-2)bc]\\
	&=&  a^2[\frac{1}{2}(b+c)^2\left( 1+t^{-2H}\right) -bc\left(1+2Ht^{-1}+H(2H-1)t^{-2}+\circ(t^{-2}) \right)\\
	&-&\frac{1}{2}(b^2+c^2)\left(1-2Ht^{-1}+H(2H-1)t^{-2}+\circ(t^{-2}) \right)   ]t^{-2H+\frac{1}{2}}\\
	&+&a^2[\frac{1}{2}(b+c)^2\left( 1+t^{-H}\right) -bc\left(1+2Ht^{-1/2}+H(2H-1)t^{-1}+\circ(t^{-1}) \right)\\
	&-&\frac{1}{2}(b^2+c^2)\left(1-2Ht^{-\frac{1}{2}}+H(2H-1)t^{-1}+\circ(t^{-1})\right)  ]t^{-4H+1}\\
	&+&[\frac{1}{2}(b+c)^2\left( 1+t^{-H}\right) -bc\left(1+2Ht^{-\frac{1}{2}}+H(2H-1)t^{-1}+\circ(t^{-1}) \right)\\
	&-&\frac{1}{2}(b^2+c^2)\left(1-2Ht^{-\frac{1}{2}}+H(2H-1)t^{-1}+\circ(t^{-1})\right)  ]\\
	&\times&[\frac{1}{2}(b+c)^2\left( 1+t^{-2H}\right) -bc\left(1+2Ht^{-1}+H(2H-1)t^{-2}+\circ(t^{-2}) \right)\\
	&-&\frac{1}{2}(b^2+c^2)\left(1-2Ht^{-1}+H(2H-1)t^{-2}+\circ(t^{-2}) \right)   ].
	\end{eqnarray*}
	First case: $0<H<\frac{1}{2}$, $a\neq0$ and $b+c\neq0$. By Taylor's expansion we get, as $t\rightarrow\infty$,
\begin{eqnarray*}
&&a^2\left(b^2+c^2-(2^{2H}-2)bc \right)t^{-4H+\frac{1}{2}}
+a^2\frac{1}{2}(b+c)^2t^{-5H+1}\\
&&+	\frac{1}{2}(b+c)^2[b^2+c^2-(2^{2H}-2)bc]t^{-3H}\\
&&\approx a^2\frac{1}{2}(b+c)^2t^{-4H+\frac{1}{2}}
+a^2\frac{1}{2}(b+c)^2t^{-5H+1}
+ \frac{1}{4}(b+c)^4t^{-3H}.
\end{eqnarray*}
Therefore
\begin{eqnarray*}
	&&a^2\left(b^2+c^2-(2^{2H}-2)bc \right)t^{-4H+\frac{1}{2}}
	+\frac{1}{2}(b+c)^2[b^2+c^2-(2^{2H}-2)bc]t^{-3H}\\
	&&\approx a^2\frac{1}{2}(b+c)^2t^{-4H+\frac{1}{2}}
	+ \frac{1}{4}(b+c)^4t^{-3H},
\end{eqnarray*}
which is true if and only if
\begin{eqnarray*}
	\frac{(b-c)^2}{2}-(2^{2H}-2)bc=0\quad \mbox{and} \quad a^2\frac{(b-c)^2}{2}-a^2(2^{2H}-2)bc=0.
\end{eqnarray*}
However, it is easy to check that $	\frac{(b-c)^2}{2}-(2^{2H}-2)bc>0$ and  $	a^2\frac{(b-c)^2}{2}-a^2(2^{2H}-2)bc>0$ for fixed $c, b$ and every real $a$.\\
	Second case: $0<H<\frac{1}{2}$, $a\neq0$ and $b+c=0$. By Taylor's expansion we get, as $t\rightarrow\infty$,
	\begin{eqnarray*}
		&&a^2[b^2+c^2-(2^{2H}-2)bc]t^{-4H+\frac{1}{2}}+a^2\left[-2Hbc+(b^2+c^2)H \right]t^{-2H-\frac{1}{2}}\\&&+\left[-2Hbc+(b^2+c^2)H \right]\times\left[b^2+c^2-(2^{2H}-2)bc\right]t^{-\frac{3}{2}}\\&&\approx  a^2\left[-2Hbc+(b^2+c^2)H \right]t^{-2H-\frac{1}{2}}+a^2\left[-2Hbc+(b^2+c^2)H \right]t^{-4H+\frac{1}{2}}\\&&+\left[-2Hbc+(b^2+c^2)H \right]^2t^{-\frac{3}{2}}.
	\end{eqnarray*}
	Hence
	\begin{eqnarray*}
		&&a^2[b^2+c^2-(2^{2H}-2)bc]t^{-4H+\frac{1}{2}}+\left[-2Hbc+(b^2+c^2)H \right]\times\left[b^2+c^2-(2^{2H}-2)bc\right]t^{-\frac{3}{2}}\\&&\approx a^2\left[-2Hbc+(b^2+c^2)H \right]t^{-4H+\frac{1}{2}}+\left[-2Hbc+(b^2+c^2)H \right]^2t^{-\frac{3}{2}} ,
			\end{eqnarray*}
	which is true if and only if $b=c=0$. This is a contradiction.\\
	Third case: $\frac{1}{2}<H<1$, $a\neq0$ and $ b-c\neq0$. By Taylor's expansion we get, as $t\rightarrow\infty$,
	\begin{eqnarray*}
		 &&a^2[b^2+c^2-(2^{2H}-2)bc]t^{-4H+\frac{1}{2}}+a^2H(b-c)^2t^{-2H-\frac{1}{2}}\\&&+H(b-c)^2[b^2+c^2-(2^{2H}-2)bc]t^{-\frac{3}{2}}\\&&\approx a^2H(b-c)^2t^{-2H-\frac{1}{2}}+
a^2H(b-c)^2t^{-4H+\frac{1}{2}}+ H^2(b-c)^4t^{-\frac{3}{2}}.
	\end{eqnarray*}
	Then
	\begin{eqnarray*}
		&&a^2[b^2+c^2-(2^{2H}-2)bc]t^{-4H+\frac{1}{2}}+H(b-c)^2[b^2+c^2-(2^{2H}-2)bc]t^{-\frac{3}{2}}\\&&\approx a^2H(b-c)^2t^{-4H+\frac{1}{2}}+ H^2(b-c)^4t^{-\frac{3}{2}},
			\end{eqnarray*}
	which is true if and only if
	\begin{eqnarray*}
		\left[b^2(1-H)+c^2(1-H)+(2-2^{2H}+2H)bc\right]=0.
	\end{eqnarray*}
	However, it is easy to check that $	b^2(1-H)+c^2(1-H)+(2-2^{2H}+2H)bc>0$ for fixed $c, b$ and every real $a$.\\	
	Fourth case: $\frac{1}{2}<H<1$, $a\neq0$  and $b-c=0$. By Taylor's expansion we get, as $t\rightarrow\infty$,
\begin{eqnarray*}
	&&a^2\left(b^2+c^2-(2^{2H}-2)bc \right)t^{-4H+\frac{1}{2}}
	+	\frac{1}{2}(b+c)^2[b^2+c^2-(2^{2H}-2)bc]t^{-3H}\\
	&&\approx  a^2\frac{1}{2}(b+c)^2t^{-4H+\frac{1}{2}}
	+\frac{1}{4}(b+c)^4t^{-3H},
\end{eqnarray*}
which is true if and only if $2-2^{2H}=0.$ This contradicts the fact that  $H\neq\frac{1}{2}$.
\end{proof}

Let us check the mixed-self-semilarity property of the mgfBm. This property was introduced in \cite{ZiliMM} for  the mfBm and investigated to show the H\"{o}lder continuity of the mfBm. See  also \cite{Coco}   for the sfBm case.
	\begin{proposition}
	For any 	$h>0$, $ \left\lbrace M^H_{ht}(a,b,c)\right\rbrace  \stackrel{\bigtriangleup}{=} \left\lbrace M^H_{t}(ah^{\frac{1}{2}}, bh^{H}, ch^{H})\right\rbrace.$\\
where $\stackrel{\bigtriangleup}{=}$ "to have the same law".
 \end{proposition}
\begin{proof}
	For fixed $h>0$, the processes $\{M_{ht}^{H}(a,b,c)\}$ and $\{M_{t}^{H}(ah^{\frac{1}{2}},bh^H,ch^H)\}$ are Gaussian and centered. Therefore, one only have to prove that they have the same covariance function.
	But, for any $s$ and $t$ in  $\mathbb{R}_+$, since $B$ and $B^H$ are independent, then
\begin{eqnarray*}
 Cov\left( M_{ht}^{H}(a,b,c),M_{hs}^{H}(a,b,c)\right)
&=&a^2h(t\wedge s)+\frac{(b^2+c^2)}{2}\left[h^{2H}\left(t^{2H}+s^{2H}-|t-s|^{2H}\right)\right]\\
&+&bc\left[h^{2H}\left(t^{2H}+s^{2H}-|t+s|^{2H}\right)\right]\\
&=&Cov\left( M_{t}^{H}(ah^{\frac{1}{2}},bh^H,ch^H),M_{s}^{H}(ah^{\frac{1}{2}},bh^H,ch^H)  \right).
\end{eqnarray*}

\end{proof}
\begin{proposition} \label{pr2}
	For all $(a, b, c)\in \mathbb{R}^{3}\setminus\{(0, 0, 0)\},$ the increments of the $M^H(a, b, c) $ are not stationary.
	\end{proposition}
	\begin{proof}
		Let $(a, b, c)\in \mathbb{R}^{3}\setminus\{(0, 0,0)\}.$
	For a fixed $t\geq0$ consider the processes  $\{P_{t},\, t\geq0\}$ define by $P_t=M_{t+s}^H(a, b, c)-M_s^H(a, b, c).$
	Using Proposition \ref{p1} we get
	\begin{eqnarray*}
		Cov(P_t,P_t)]
		&=&E\left[(M_{t+s}^H(a, b, c)-M_s^H(a, b, c))^2 \right]\\
		&=&a^2(t+s+s)-2^{2H}bc(({t+s})^{2H}+s^{2H})-2a^2 s\\&+&(b^2+c^2)|t+s-s|^{2H}+2bc|t+s+s|^{2H} \\
		&=& a^2(t+2s)-2^{2H}bc(({t+s})^{2H}+s^{2H})-2a^2 s+(b^2+c^2)t^{2H}+2bc|t+2s|^{2H} .
	\end{eqnarray*}
	Using Proposition \ref{p1} we get
	\begin{eqnarray*}
		Cov(M_t^H(a, b, c), M_t^H(a, b, c))=a^2t+\left(b^2+c^2-(2^{2H}-2)bc \right)t^{2H}. 	
	\end{eqnarray*}
	Since  both processes are centered Gaussian, the inequality of covariance functions implies that $P_t$ does not have the same distribution as $M_t^H(a, b, c)$. Thus, the incremental behavior of $M^H(a, b, c)$ at any point in the future is not the same. Hence  the  increments of $M^H(a, b, c)$ are not stationary.
\end{proof}
\begin{remarks}
As a consequence of Proposition \ref{pr2}, we see that:
	\begin{enumerate}	
		\item the increments of $M^{H}(0, b, c)$   are not stationary for all $( b, c)\in \mathbb{R}^{2}\setminus\{(0, 0)\}.$
		\item  the increments of $M^{H}(a, b, 0)$  are stationary for all $( a, b)\in \mathbb{R}^{2}.$
	\end{enumerate}
	\end{remarks}
\begin{proposition}
	\begin{enumerate}	
		\item
Let $H\in (0, 1).$ The mgfBm admits a version whose sample paths are almost H\"{o}lder continuous of order strictly less than $\frac{1}{2}\wedge H.$
		\item When $b$ or $c$ not zero  and $H\in(0, 1)\setminus\{ \frac{1}{2}\}$   the mgfBm is not a semi-martingale.
\end{enumerate}	
	
\end{proposition}
\begin{proof}
	\begin{enumerate}	
		\item
Let $s$ and $t$ in  $\mathbb{R}_+$ and $ \alpha=2.$ The proof follows by Kolmogorov criterion from Lemma 3  in \cite{Zili} and using  Proposition \ref{p1} we get
\begin{eqnarray*}
		E\left( |M_t^H(a,b,c)-M_s^H(a,b,c)|^ \alpha\right) &=&a^2|t-s|-2^{2H}bc(t^{2H}+s^{2H})\\&+&(b^2+c^2)|t-s|^{2H}+2bc|t+s|^{2H}\\&\leq& C_{ \alpha} |t-s|^{ \alpha(\frac{1}{2}\wedge H)}\
	\end{eqnarray*}
where  $C_{ \alpha } =\left( a^2 +\nu(b,c, H)\right)$ and $\nu_{(b,c, H)}$  is given in Lemma 3  in \cite{Zili}.

	\item Suppose first that  $H<\frac{1}{2}$. We get from Proposition  \ref{p1}
\begin{eqnarray*}
E \left(M_t^H(a,b,c)-M_s^H(a,b,c)\right)^2 \geq \gamma_{(b,c, H)}(t-s)^{2H}.
\end{eqnarray*}
Since $2H<1$ and $\gamma_{(b,c, H)}>0$ then the assumption of Corollary 2.1 in \cite{BGT} is satisfied, and consequently  the mgfBm is not a semi-martingale.\\
Suppose now that  $H>\frac{1}{2}$. We get from Properties  \ref{p1}
\begin{eqnarray*}
a^{2}(t-s)+ \gamma_{(b,c, H)}(t-s)^{2H}\leq E \left(M_t^H(a,b,c)-M_s^H(a,b,c)\right)^2 \leq (a^{2}+\nu_{(b,c, H)})(t-s)^{1\wedge2H}
\end{eqnarray*}
then
\begin{eqnarray*}
 \gamma_{(b,c, H)}(t-s)^{2H}\leq E \left(M_t^H(a,b,c)-M_s^H(a,b,c)\right)^2 \leq (a^{2}+\nu_{(b,c, H)})(t-s)^{2H}.
\end{eqnarray*}
Since $1<2H<2$  and $\nu_{(b,c, H)}>0$ then the assumption of Lemma 2.1 in \cite{BGT} is satisfied, and consequently  the mgfBm is not a semi-martingale.
\end{enumerate}	
\end{proof}


\section{ Long range dependence of the mgfBm }

 \begin{definition}
	We say that the increments of a stochastic process $X$ are long-range dependent if for every integer $p\geq1$, we have
	\begin{eqnarray*}
		\sum_{n\geq 1}R_X(p,p+n)=\infty,
	\end{eqnarray*}
	where $R_X(p,p+n)=E\left((X_{p+1}-X_p)(X_{p+n+1}-X_{p+n}) \right). $
\end{definition}
This property was investigated in many papers (e.g. \cite{AlexY, Casa, ChenXuHu, BardBa, GaoH}).
\begin{theorem}
	For every $a\in  \mathbb{R} $ and $(b,c)\in \mathbb{R}^2 \setminus\left\lbrace (0,0)\right\rbrace $ the increments of $M^H(a,b,c)$ are long-range dependent if and only if $H>\frac{1}{2}$ and $b\neq c$ .
	\end{theorem}
	\begin{proof}
		For all $n\in \mathbb{N}\setminus\{0\}$ and $p\geq 1$ we have
		\begin{eqnarray*}
			R_{M}(p,p+n)&=&E\left((M_{p+1}^H(a, b, c)-M_p^H(a, b, c))(M_{p+n+1}^H(a, b, c)-M_{p+n}^H(a, b, c)) \right)\\
			&=&E\left(M_{p+1}^H(a, b, c)M_{p+n+1}^H(a, b, c)\right)-E\left(M_{p+1}^H(a, b, c)M_{p+n}^H(a, b, c) \right)\\
			&-&E\left(M_{p}^H(a, b, c)M_{p+n+1}^H(a, b, c)\right)+E\left(M_{p}^H(a, b, c)M_{p+n}^H(a, b, c) \right)\\
			&=&C({p+1},{p+n+1})-C({p+1},{p+n})-C({p},{p+n+1})+C({p},{p+n})\\
			&=& a^2(p+1) +\frac{(b+c)^2}{2}\left( (p+1)^{2H}+(p+n+1)^{2H}\right) -bc(2p+n+2)^{2H}\\&-&\frac{(b^2+c^2)}{2}n^{2H}
			-a^2(p+1)-\frac{(b+c)^2}{2}\left( (p+1)^{2H}+(p+n)^{2H}\right) \\&+&bc(2p+n+1)^{2H}+\frac{(b^2+c^2)}{2}|n-1|^{2H}
			-a^2p\\&-&\frac{(b+c)^2}{2}\left( (p)^{2H}+(p+n+1)^{2H}\right) +bc(2p+n+1)^{2H}+\frac{(b^2+c^2)}{2}|n+1|^{2H}.
\end{eqnarray*}
Hence
\begin{eqnarray*}
			R_{M}(p,p+n)&=&
			\frac{(b^2+c^2)}{2}\left( (n+1)^{2H}-2n^{H}+(n-1)^{2H}\right)\\&-&bc\left((2p+n+2)^{2H}-2(2p+n+1)^{2H}+(2p+n)^{2H} \right).
		\end{eqnarray*}
		Then for every integer $p\geq1$, by Taylor's expansion, as $n\rightarrow\infty$, we have
		\begin{eqnarray*}
			R_M(p,p+n)&=&\frac{b^2+c^2}{2}n^{2H}\left[ \left( 1+\frac{1}{n}\right)^{2H}-2+\left( 1-\frac{1}{n}\right)^{2H} \right]
			\\&-&bcn^{2H}\left[ \left( 1+\frac{2p+2}{n}\right)^{2H}-2 \left( 1+\frac{2p+1}{n}\right)+\left( 1+\frac{2p}{n}\right)^{2H} \right] \\
			&=&H(2H-1)n^{2H-2}(b-c)^2\\&-&4H(2H-1)(H-1)bc(2p+1)n^{2H-3}(1+\circ(1)).
		\end{eqnarray*}
		If $b\neq c$, we see that as $n\rightarrow\infty$,
		\begin{eqnarray*}
			R_M(p,p+n)\approx H(2H-1)n^{2H-2}(b-c)^2.
		\end{eqnarray*}
		Then
		\begin{eqnarray*}
			\sum_{n\geq 1}R_M(p,p+n)=\infty \ \ \Leftrightarrow \ 2H-2>-1 \ \ \Leftrightarrow H>\frac{1}{2}.
		\end{eqnarray*}
		If $b=c$, then, as $n\rightarrow\infty$,
		\begin{eqnarray*}
			R_M(p,p+n)\approx 4H(2H-1)(H-1)a^2(2p+1)n^{2H-3}.
		\end{eqnarray*}
		For every $H\in(0,1) $, we have $2H-3<-1$ and, consequently,
		\begin{eqnarray*}
			\sum_{n\geq 1}R_M(p,p+n)<\infty.
		\end{eqnarray*}

\end{proof}
\begin{remarks}

	\begin{enumerate}	
		\item  For all $a \in \mathbb{R}$ and $ b\in \mathbb{R}\setminus\{0\}$ the increments of $M^H(a,b,0)$ are long-range dependent  if and only if $H>\frac{1}{2}.$
	\item If $b=c=\frac{1}{\sqrt{2}}$ the increments of $M^H(0,\frac{1}{\sqrt{2}},\frac{1}{\sqrt{2}})$ are short-range dependent if and only if $H\in(0, 1).$ But if $b\neq c$
the increments of $M^H(0,b,c)$ are long-range dependent  if and only if $H>\frac{1}{2}.$

       \item From  \cite{Bojd},  the increments of $M^{H}(0, \frac{1}{\sqrt{2}}, \frac{1}{\sqrt{2}})$ on intervals $[u,u+r],[u+r,u+2r]$ are more weakly correlated than those of $M^{H}(0, 1, 0)$.
       \item From  \cite{ZiliM}, If $H>\frac{1}{2}$, $b^2+c^2=1$ and $bc\geq0$, the increments of $M^{H}(0, b, c)$ are more weakly correlated than those of $M^{H}(0, 1, 0)$, but more strongly correlated than those of $M^{H}(0, \frac{1}{\sqrt{2}}, \frac{1}{\sqrt{2}})$.
       \item From  \cite{ZiliM}, If $H\geq\frac{1}{2}$, $(bc\leq0 \ and \ (b-c)^2\leq1)$ or $(bc\geq0 \ and \ b^2+c^2\leq1)$, the increments of $M^{H}(0, b, c)$ are more strongly correlated than those of both $M^{H}(0, 1, 0)$ and $M^{H}(0, \frac{1}{\sqrt{2}}, \frac{1}{\sqrt{2}})$.

   \end{enumerate}
\end{remarks}


\newpage


\end{document}